\documentclass[12pt,a4paper,reqno]{amsart}
\usepackage{latexsym}
\usepackage{amssymb}
\usepackage{enumitem}

\usepackage[active]{srcltx}

\def \zmontar{\buildrel}

\def \za{\alpha}

\def \zg{\gamma}
\def \zd{\delta}
\def \ze{\varepsilon}

\def \zl{\lambda}
\def \zm{\mu}

\def \zx{\xi}

\def \zp{\pi}

\def \zr{\rho}

\def \zs{\sigma}

\def \zt{\tau}

\def \zf{\varphi}

\def \zq{\psi}
\def \zw{\omega}
\def \zG{\Gamma}

\def \zF{\Phi}

\def \zlma{\ell}

\def \zsu{\sum}

\def \zin{\cap}

\def \zun{\cup}
\def \zung{\bigcup}

\def \zte{\otimes}

\def \zmm{\pm}

\def \zpu{\cdot}
\def \zpor{\times}
\def \zci{\circ}

\def \zmei{\leq}
\def \zmai{\geq}
\def \zco{\subset}

\def \zpe{\in}

\def \zeq{\equiv}

\def \znoi{\neq}

\def \znoco{\not\subset}

\def \znope{\not\in}

\def \zpar{\partial}
\def \zinf{\infty}
\def \zva{\emptyset}

\def \zfl{\rightarrow}

\def \zbv{|}
\def \zdbv{||}
\def \z/{\over}

\newcommand {\CC}{\mathbb C}

\newcommand {\RR}{\mathbb R}

\newcommand {\NN}{\mathbb N}
\newcommand {\A}{\mathcal A}

\newcommand {\F}{\mathcal F}
\newcommand {\co}{\colon}

 \def\mylabel#1{\label{#1}}   

\newtheorem{theorem}{Theorem}[section]
\newtheorem*{theorem*}{Theorem}
\newtheorem{lemma}[theorem]{Lemma}
\newtheorem{corollary}[theorem]{Corollary}
\newtheorem*{corollary*}{Corollary}
\newtheorem{proposition}[theorem]{Proposition}

\newtheorem*{example*}{Example}

\theoremstyle{definition}
\newtheorem{remark}[theorem]{Remark}
\newtheorem{example}[theorem]{Example}

\newcommand{\Aut}{\operatorname{Aut}}
\newcommand{\Diff}{\operatorname{Diff}}
\newcommand{\Homeo}{\operatorname{Homeo}}

\title[Finite groups of diffeomorphisms determined by a vector field]{Finite groups of diffeomorphisms are topologically determined by a vector field}

\author{F.J.~Turiel}
\address[F.J.~Turiel]{
Departamento de {\'A}lgebra, Geometr{\'\i}a y Topolog{\'\i}a,
Facultad de Ciencias,
Campus de Teatinos, s/n,
29071-M{\'a}laga, Spain}
\email[F.J.~Turiel]{turiel@uma.es}

\author{A.~Viruel}
\address[A.~Viruel]{
Departamento de {\'A}lgebra, Geometr{\'\i}a y Topolog{\'\i}a,
Facultad de Ciencias,
Campus de Teatinos, s/n,
29071-M{\'a}laga, Spain}
\email[A.~Viruel]{viruel@uma.es}

\thanks{Authors are partially supported by
MEC-FEDER grant MTM2016-78647-P, and JA grant FQM-213}

\begin{document}

\begin{abstract}
In a previous work it is shown that every finite group $G$ of diffeomorphisms
of a connected smooth manifold $M$ of dimension $\geq 2$ equals, up to
quotient by the flow, the centralizer of the group of smooth automorphisms
of a $G$-invariant complete vector field $X$ (shortly $X$ describes $G$).
Here the foregoing result is extended to show that every finite group of
diffeomorphisms of $M$ is described, {\em within the group of all
homeomorphisms of $M$}, by a vector field.

As a consequence, it is proved that a finite group of homeomorphisms of
 a compact connected topological $4$-manifold, whose action is free,
 is described by a continuous flow.
\end{abstract}

\maketitle

\section{Introduction} \mylabel{sec-1}

The study of the automorphism group, or centralizer, of a complete vector field $X$ of class
$C^r$, $r\zmai 1$, or more precisely that of its quotient by the flow of $X$, is a classical question
with a great amount of interesting results. Often these quotient groups are trivial or almost trivial if
a reasonable  hypothesis of transversality is imposed.

Therefore it is natural to consider the inverse point of view (the inverse Galois problem):
given a group of diffeomorphisms $G$ do construct a complete vector field $X$ whose automorphism group, up to quotient by the flow, equals $G$ (shortly one will say that $X$ determines
or describes $G$). Notice that this last problem can be  addressed in topological manifolds
and homeomorphisms by replacing the vector field by a continuous flow.

In \cite{TV} it is shown that every finite group of diffeomorphisms of a smooth connected
manifold is determined, within the group of all diffeomorphisms, by a vector field. Here
the foregoing result is extended to show that every finite group of diffeomorphisms can be
described, {\em within the group of all homeomorphisms}, by a vector field.

\begin{theorem}\mylabel{mainth}
Let $M$ be a connected $C^{\zinf}$ manifold  of dimension $m\zmai 2$,
and $G$ be a finite subgroup  of diffeomorphisms of $M$. Then there exists $X$, a complete $G$-invariant vector field on $M$, such that the map
$$\begin{array}{rcl}
G\zpor\RR & \to & \Aut_{0}(X)\\
(g,t) & \mapsto & g\zci\zF_t
\end{array}$$
is a group isomorphism, where $\zF$  and $\Aut_{0}(X)$ denote the flow
and the group of continuous automorphisms of $X$ respectively.
\end{theorem}

An immediate consequence of Theorem \ref{mainth} is that any smoothable finite group of homeomorphisms of a connected topological
manifold is determined by a continuous flow.

\begin{corollary}\mylabel{mainco}
Let $G$ be a smoothable finite group of homeomorphisms of a connected topological manifold
$E$ of dimension $\zmai 2$. Then there  exists a $G$-invariant continuous flow
 $\zq\co\RR\zpor E\zfl E$  such that the map
$$\begin{array}{rcl}
G\zpor\RR & \to & \Aut_{0}(\zq)\\
(g,t) & \mapsto & g\zci\zq_t
\end{array}$$
is a group isomorphism, where  $\Aut_{0}(\zq)$ denotes the group of continuous automorphisms
of $\zq$.
\end{corollary}

\begin{remark}\mylabel{mainre}
While it is a classical result that any finite group of homeomorphisms of a compact surface is
smoothable, the situation in dimension three is not straightforward: not every finite group $G$ of homeomorphisms of a
$3$-dimensional topological compact manifold is smoothable \cite{ BI}. Indeed $G$ is smoothable if and only if it is locally linear \cite[Theorem 2.1 and Remark 2.4]{KL}.

Therefore the corollary above applies to connected compact surfaces and, in the case of topological connected compact $3$-manifolds, if $G$ is locally linear.

For a generalization of Corollary \ref{mainco} to some cases of non-smoothable actions see
Example \ref{exaSph} and Theorem \ref{exaTop}. In this last one we show that a finite group of
homeomorphisms of a compact connected topological $4$-manifold, whose action is free, is described by a continuous flow.
\end{remark}

\noindent \textbf{Terminology:} We assume the reader is familiarized with our previous
paper \cite{TV}. All structures and objects considered in this work are smooth, i.e.\ real
$C^{\zinf}$, and manifolds are without boundary, unless another thing
is stated. Whenever we say a set is countable we mean the set is either a finite set or a countably infinite set.
For the general questions on Differential Geometry the reader is
referred to \cite{KN} while we refer to \cite{HI} for basic facts on Differential Topology.

\section{Some preliminary notions}\mylabel{sec-2}

Given a vector field $Z$ on an $m$-manifold $M$, a {\em continuous automorphism of} $Z$
is a homeomorphism $f\colon M\zfl M$ which maps integral curves of
$Z$ into integral curves of $Z$ (i.e.\ if $\zg(t)$ is an integral curve
of $Z$ then $f(\zg(t))$ is so.) The set $\Aut_{0}(Z)$ of all continuous automorphisms
of $Z$ is a subgroup of the group of homeomorphisms of $M$.
If $Z$ is complete and $\zF_t$ denotes its flow, then $f\zpe\Aut_0 (Z)$ if and only if
$f\zci\zF_t =\zF_t \zci f$ for any $t\zpe\RR$.

In a more general setting, given a topological space $E$ a {\em continuous flow} is
a continuous map $\zq\co\RR\zpor E\zfl E$ such that $\zq_0 =Id$ and
$\zq_{t+s}=\zq_t \zci\zq_s$ for each $t,s\zpe\RR$. As before, $\Aut_0 (\zq)$ is the group of all
homeomorphisms  $f\co E\zfl E$ such that $f\zci\zq_t =\zq_t \zci f$, $t\zpe\RR$. We say
that a subset $S$ of $\Homeo(E)$ is {\em smoothable} if there exists a structure of smooth
manifold on $E$ which is compatible with the preexisting topology and makes every element of $S$
a diffeomorphism.

Returning to the smooth framework again, given a vector field $Z$ on an $m$-manifold $M$, a {\em pseudo-circle} of $Z$ is a
subspace of $M$ which is homeomorphic to $S^1$ and consists
of a regular trajectory of $Z$ and an isolated singularity. In this case the $\za$-limit
and the $\zw$-limit of the regular trajectory is the singular point.

Let $B(r)$ be the open ball in $\RR^m$ centered at the origin and
radius $r>0$. For the purpose of this work, we will say that $p\zpe M$ is a
{\em source} of $Z$ if there exists an open neighborhood of this point which is diffeomorphic to
an open ball $B(r)$, with $p\zeq 0$, such that in the coordinates given by the
diffeomorphism
$$ Z=\zf\zpu\left(\zsu_{j=1}^m x_j {\frac {\zpar} {\zpar x_j}}\right)$$
where
\begin{enumerate}[label={\rm (\arabic{*})}]
\item  $\zf$ is a non-negative function and $\zf^{-1}(0)$ is countable, and

\item\mylabel{especial1} on each ray issuing from the origin there are at most a finite number
of zeros of $\zf$.
\end{enumerate}

By condition \ref{especial1}, for every ray issuing from the origin there exists just one regular
trajectory whose $\za$-limit is $p\zeq 0$ and that near the origin lies along this ray.

A point $q\zpe M$ is called a {\it rivet} if the following hold:
\begin{enumerate}[label={\rm (\alph{*})}]
\item $q$ is an isolated singularity of $Z$,

\item\mylabel{especial3} around $q$ one has $Z=\zq{\tilde Z}$ where $\zq$ is a function and ${\tilde Z}$ a vector field with ${\tilde Z}(q)\znoi 0$, and

\item\mylabel{especial3a} no trajectory has $q$ as $\za$-limit and
$\zw$-limit at the same time.
\end{enumerate}

Note that by \ref{especial3} and \ref{especial3a}, any rivet is the $\zw$-limit of exactly one
regular trajectory, the $\za$-limit of another different one and moreover, it is an isolated
singularity of index zero.

A {\em topological rivet} means an isolated singularity of $Z$ that is the $\za$-limit
of a single regular trajectory, the $\zw$-limit of another single regular trajectory and
both trajectories are different. As one would expect, any rivet is a topological rivet.

By definition, a {\it chain} of $Z$ is a finite and ordered
sequence of three or more different regular trajectories, each of
them called a {\it link}, such that:

\begin{enumerate}[label={\rm (\alph{*})}]
\item The $\za$-limit of the first link is a source or empty.

\item The $\zw$-limit of the last link is a pseudo-circle.

\item Between two consecutive links the $\zw$-limit of the
first one equals the $\za$-limit of the second one. Moreover this
set consists in a rivet.
\end{enumerate}

The  number of links defining a chain is called the {\it order of the chain}. The {\it
$\zw$-limit} of a chain is that of its last link.

Given a subset $Q\subset M$, we say that {\it the dimension of $Q$  does not exceed
 $\zlma$}, or $Q$ {\it can be enclosed in dimension
$\zlma$}, if there exists a countable collection $\{
N_{\zl}\}_{\zl\zpe L}$ of submanifolds of $M$, all of them of
dimension $\zmei\zlma$, such that $Q\zco \zung_{\zl\zpe
L}N_{\zl}$. Note that the countable union of sets whose dimension
does not exceed $\zlma$, does not exceed dimension $\zlma$
too. On the other hand, if the dimension of $Q$  does not exceed $\zlma<m$ then $Q$ has measure zero and therefore $Q$ has
empty interior.

Let us give the last definition of this section. A vector field $Z$ on $M$
is called {\em limit} (abbreviation of ``with an almost controlled
$\zw$-limit'') if the following conditions hold:

\begin{enumerate}[label={\rm (\roman{*})}]
\item The set of zeros of $Z$ is discrete (that is with no accumulation point).
\item $Z$ has exactly one pseudo-circle.
\item\mylabel{especial2} There exists a set $Q\zco M$ whose dimension does not exceed $m-1$
such that the trajectory of any point of $M-Q$ has the
pseudo-circle as $\zw$-limit.
\item $Z$ has no chain and no periodic regular trajectory.
\end{enumerate}

By \ref{especial2} the union of the trajectory of any point of $M-Q$ and the pseudo-circle
is a connected set, hence the $Z$-saturation of $M-Q$ together with the pseudo-circle
is a connected set too. Therefore $M=\overline{M-Q}$ is connected.
Moreover $dim\, M\zmai 2$, otherwise $M$ equals the pseudo-circle  and
\ref{especial2} cannot hold.

\begin{proposition}\mylabel{pro-1}
Each sphere $S^k$, $k\zmai 2$, supports a limit vector field.
\end{proposition}

\begin{proof}
On $S^{k}\zco\mathbb R^{k+1}$ consider the vector field
$$\zx=-\zsu_{j=1}^{k}x_{j}x_{k+1}(\zpar/\zpar x_{j})+
(1-x_{k+1}^{2})(\zpar/\zpar x_{k+1}),$$ orthogonal projection of the vector field $\zpar/\zpar x_{k+1}$ onto
the sphere, whose
trajectories go from the south pole to the north one. Since $\zx$
is transverse to the equator $E=\{ x\zpe S^{k}\zbv x_{k+1}=0\}$,
one may identify an open neighborhood $A$ of $E$ to
$(-\ze,\ze)\zpor S^{k-1}$, endowed with coordinates
$(t,y)=(t,y_{1},\dots,y_{k})$ where $S^{k-1}\zco {\mathbb R}^{k}$,
in such a way that $E$ corresponds to $\{0\}\zpor S^{k-1}$ and
$\zx=\zpar/\zpar t$.  Thus the north band is given by $t>0$ and the
south one by $t<0$.

Let $\zf:{\mathbb R}\zfl{\mathbb R}$ be a function such that  $\zf\big((-\zinf,-\ze/2]\big)=1$, $\zf\big([\ze/2,\zinf))=-1$, and $\zf(t)=0$ if and only if $t=0$.

First assume $k=2$. On $A$ consider the vector field
$Z'=\zf(t)\zpar/\zpar t+(1-\zf^{2}(t))(-y_{2}\zpar/\zpar y_{1}+y_{1}\zpar/\zpar
y_{2})$ and extend it outside $A$ by $-\zx$ on the north part and by
$\zx$ on the south one. Fixed a point $p_{0}\zpe E$, consider a
function $\zq:S^{2}\zfl {\mathbb R}$ vanishing at $p_{0}$ and
positive on $S^{2}-\{p_{0}\}$. It is easily checked that $Z=\zq Z'$
is a limit vector field (here $Q$ is the equator plus both
poles, thus the dimension of $Q$ does not exceed $1$, and observe that the only sources are the poles).

Now assume $k\zmai 3$; let ${\widetilde Z}$ be a limit vector field on
$S^{k-1}$ constructed by induction. On $A$ consider the vector
field $Z=\zf(t)\zpar/\zpar t+(1-\zf^{2}(t)){\widetilde Z}$, where
${\widetilde Z}$ is regarded as a vector field on $A$ in the obvious
way, that is tangent to the second factor, and extend it outside
of $A$ by $-\zx$ on the north part and by $\zx$ on the south one. We now prove that
$Z$ is a limit vector field on $S^{k}$.

First note that the only sources of $Z$ are the poles. Moreover,
$Z$ does not have any rivet, which implies that $Z$ has no chain.
Indeed clearly no point of $S^k -E$ is a rivet; on the other hand
if $(0,q)$ is a rivet, as $Z$ is tangent to
$\{0\}\zpor S^{k-1}$ the only trajectory whose $\zw$-limit is this point
is included in $\{0\}\zpor S^{k-1}$. But clearly $(-\zd,0)\zpor\{q\}$ for
a $\zd>0$ sufficiently small is included in a trajectory with $\zw$-limit
$(0,q)$ what leads to {\em contradiction}.

By construction, no trajectory in $S^{k}-E$ is regular and periodic,
so $Z$ does not possess any periodic regular trajectory. On the
other hand, if ${\widetilde Z}$ is regarded as a vector field on $E$
and ${\widetilde Q}\zco E$ satisfies condition \ref{especial2} for ${\widetilde Z}$,
it suffices to take as $Q$ the union of all trajectories of $\zx$ passing
through ${\widetilde Q}$ plus both poles. For if $p\zpe (S^{n}-E)$
and $q\zpe E$ belong to the same $\zx$-trajectory, then the
$\zw$-limit of the $Z$-trajectory of $p$ equals the $\zw$-limit of
the ${\widetilde Z}$-trajectory of $q$.
\end{proof}

\section{The almost free case}\mylabel{sec-3}
Let $M$ be a connected manifold of dimension $m$. Given a
diffeomorphism $\zf\colon M\zfl M$, the {\em isotropy} of $\zf$ is the set
$I_\zf \colon =\{p\zpe M\colon \zf(p)=p\}$. A point $p\zpe I_\zf$
is said to be {\em positive} or {\em negative} according to the sign of
the determinant of $\zf_* (p)\colon T_p M\zfl T_p M$.
Obviously $I_\zf =I_{\zf}^+ \zun I_{\zf}^-$ where the
{\em positive isotropy} $I_{\zf}^+$ is the set of positive points
and the {\em negative isotropy} $I_{\zf}^-$ that of negative ones.

If $\zf\znoi Id$ has finite order, that is to say $\zf$ spans a finite
subgroup of diffeomorphisms, then $\zf$ is an isometry for some
Riemannian metric. Thus if $p\zpe I_\zf$, the use of normal coordinates with
origin $p$ allows us to identify the diffeomorphism $\zf$ with an element of
$O(m)$ different from the identity. Therefore locally $I_{\zf}^-$ is a
regular submanifold of codimension $\zmai 1$ and $I_{\zf}^+$ a
regular submanifold of codimension $\zmai 2$.

By definition the {\em maximal  isotropy} $I_{\zf}^{\max}$ is the set
of those points $p\zpe I_{\zf}^{-}$ such that the codimension of
$ I_{\zf}^{-}$ at $p$ equals $1$. It is easily seen that $I_{\zf}^{\max}$
is either empty or a closed regular submanifold of codimension $1$.
Notice that if $p\zpe I_{\zf}^{\max}$ then in normal coordinates with origin
$p$ the diffeomorphism $\zf$ is a symmetry with respect to a hyperplane
(the trace of $I_{\zf}^{\max}$).

Let $G$ be a finite group of diffeomorphisms of $M$. Let $e\in G$ be the identity element and
$\zlma$ be the order of $G$. By  the {\em isotropy}, the {\em positive isotropy}, the
{\em negative isotropy} and the {\em maximal isotropy} of $G$ we mean
$$I_G^{*}=\zung_{g\zpe G-\{e\}}I_\zf^{*}$$
where $*$ equals nothing, $+$, $-$ or $\max$ respectively.

For the purpose of this work one will say that the action of $G$ is
{\em almost free} if $I_G=I_G^{\max}$.

\begin{lemma}\mylabel{lem-1}
Assume $I_G^{+}=\zva$. Then the following hold:
\begin{enumerate}[label={\rm (\alph{*})}]
\item\mylabel{especial4} If $g$ and $h$ are two different elements of $G-\{e\}$ then
$I_{g}^{-}\zin I_{h}^{-} =\zva$.
\item\mylabel{especial5} If $I_{g}^{-}\znoi\zva$ then $g^2=e$.
\end{enumerate}
\end{lemma}

\begin{proof}
\ref{especial4} If $p\zpe I_{g}^{-}\zin I_{h}^{-}$  then $p$ belongs to
 the positive isotropy of $gh^{-1}$, so $gh^{-1}=e$.

\ref{especial5} If $p\zpe  I_{g}^{-}$ then $p$ belongs to $I_{g^2}^+$,
hence $g^2 =e$.
\end{proof}

{\em In the remainder of this section the action of $G$ is assumed to be almost
free}. Our goal will be to prove the main theorem under this
supplementary hypothesis.

The proof consists of four steps. In the first one, we
construct a vector field $Z$ as the gradient of a suitable $G$-invariant Morse
function $\zm$. In a second step, one modifies $Z$ for obtaining a new
$G$-invariant vector field $Y$ with as many pseudo-circle as (local) minima
of $\zm$.

The third part is the construction from $Y$ of a $G$-invariant vector field $X$
that possesses a countable family of chains. These chains are topological
invariant of $X$ and allow us to control its continuous automorphisms.
Finally, the fourth step is devoted to determine these automorphisms.

\subsection{The gradient vector field}\mylabel{sec-3.1}
Let $\zm\colon M\zfl{\mathbb R}$ be a Morse function that is
$G$-invariant, proper and non-negative, whose existence is assured
by a result of Wasserman \cite{WA}. Let $C$ denote the set of
critical points of $\zm$, which is closed, discrete and countable. As $M$ is paracompact,
there exists a locally finite family of disjoint
open sets $\{A_{p}\colon p\zpe A_{p}\}_{p\zpe C}$ which is $G$-invariant, i.e.\
$A_{g\zpu p}=g\zpu A_p$ for any $p\zpe C$ and any $g\zpe G$.
By shrinking  each $A_p$ if necessary, one constructs a collection of charts
$\{(A_{p},\zr_{p})\}_{p\zpe C}$ such that:

\begin{enumerate}[label={\rm (C.\arabic{*})}]
 \item\mylabel{especial6a} $\zr_p (A_p )=B(2r_p )$ for some $r_p >0$ and
$\zr_p (p)=0$.
\item\mylabel{especial6b} $\zm=\zsu_{j=1}^{k}x_{j}^{2}
-\zsu_{j=k+1}^{m-1}x_{j}^{2}+\ze x_{m}^2 +\zm(p)$ on $A_p$
where $\ze=\zmm 1$. (Of course $k$ and $\ze$ depend on $p$, and
$x=(x_1 ,\dots,x_m)$ are the coordinates associated to the chart
$(A_{p},\zr_{p})$. Nevertheless, in order to avoid an over-elaborated
notation, these facts are not indicated unless it is completely necessary.)
\item\mylabel{especial6c} $\zr_{g\zpu p}\zci g\zci\zr_{p}^{-1}$ equals
\begin{itemize}
\item the identity map, if the $G$-orbit of $p$ has exactly $\zlma$ elements, or
\item the identity
or the symmetry $\zG(x)=(x_1 ,\dots,x_{m-1},-x_m )$, if the $G$-orbit of $p$ has less
than $\zlma$ elements.
\end{itemize}

\item\mylabel{especial6d} The function $\zm$ can be chosen in such a way
that the $G$-orbit of every (local) minimum has $\zlma$ elements.
\end{enumerate}

Indeed, let $\mathcal O\subset C$ be a $G$-orbit and fix a point
$p\in {\mathcal O}$. If $\vert\mathcal O\vert=\zlma$ one constructs such a chart
around $p$ and then use the $G$-action to get a chart around any point of $\mathcal O$.

If $\vert\mathcal O\vert<\zlma$ then $p\zpe I_G$ and, by Lemma \ref{lem-1},
$\vert\mathcal O\vert=\zlma /2$ and there exists just one element $h\zpe G-\{e\}$ such that
$h\zpu p=p$. Moreover $h^2 =e$. As $I_h =I_{h}^{\max}$ there are coordinates
$(y_1 ,\dots,y_m)$ around $p\zeq 0$ such that $h$ is given by the symmetry
$\zG(y)=(y_1 ,\dots,y_{m-1},-y_m )$.

By Lemma \ref{leA-1} applied to coordinates $(y_1 ,\dots,y_m)$ (observe that now
$x$ and $y$ have exchanged their roles) there exist
coordinates $(x_1 ,\dots,x_m)$ around $p\zeq 0$ in which $h$ is still given by
the symmetry $\zG(x)=(x_1 ,\dots,x_{m-1},-x_m )$ and
$\zm=\zsu_{j=1}^{k}x_{j}^{2}
-\zsu_{j=k+1}^{m-1}x_{j}^{2}+\ze x_{m}^2 +\zm(p)$, $\ze=\zmm 1$.
Now for another point $q\zpe\mathcal O$ choose a $g\zpe G-\{e\}$ such that
$q=g\zpu p$ and set $\zr_q =\zr_p \zci g^{-1}$.

Finally if our $p$ is a minimum of $\zm$, always with a $G$-orbit of
$\zlma /2$ elements, by applying Proposition \ref{prA-1} to a
$0<r'_p <min\{1,r_p \}$ we may modify $\zm$ inside $\zr^{-1}(B(r_p ))$
to construct a new $h$-invariant Morse function, still called $\zm$, such that
each one of its minima in $A_p$ is not $h$-invariant and, as before, transfer
this modification to every $A_q$, $q\zpe\mathcal O-\{p\}$, by means of a
$g\zpe G-\{e\}$ such that $g\zpu p=q$.

As every minimum in $A_p$ of the new $\zm$ is not $h$-invariant, then the $G$-orbit
of any minimum in $\zung_{q\zpe\mathcal O}A_q$ of the new $\zm$ has
$\zlma$ elements. Obviously the same thing can be done with any other
$G$-orbit with $\zlma/2$ elements consisting of minima.

On the other hand since $\zbv\zt(x)\zbv\zmei\zdbv x\zdbv^2$ in Proposition
\ref{prA-1}, the new function $\zm$ is proper and low bounded by $-1$
(more exactly the difference between the new function $\zm$ and the old
one takes its values in $[-2,0]$). Therefore replacing $\zm$ by $\zm+1$
shows \ref{especial6d}.

On $M$ there always exists a Riemannian metric $g'$ that on each
$\zr_{p}^{-1}(B(r_p ))$ is written as $2\zsu_{j=1}^m dx_j \zte dx_j$. Therefore
shrinking every $A_p$ allows to assume $g'=2\zsu_{j=1}^m dx_j \zte dx_j$ on the
whole $A_p$. Moreover taking into account Property \ref{especial6c} of the
collection $\{(A_{p},\zr_{p})\}_{p\zpe C}$ we may assume, without losing
the property  above, that $g'$ is $G$-invariant by considering
$(1/\zlma)\zsu_{h\zpe G}h^* (g')$ instead of $g'$ if necessary.

Let $Z'$ be the gradient vector field of $\zm$ with respect to $g'$ and
$\zf\co M\zfl\RR$ be a $G$-invariant proper function that is constant around every $p\zpe C$.
As before, $\zf$ can be supposed constant on each $A_p$ by shrinking these open sets
if necessary. It is well known that the vector field
$Z=e^{-(Z'\zf)^2}Z'$
is complete. Moreover $Z$ is the gradient of $\zm$ with respect to the $G$-invariant
Riemannian metric $\tilde g =e^{(Z'\zf)^2}g'$.

On the other hand $\tilde g =g'$ on every $A_p$, $p\zpe C$, since $\zf$ is constant
on these sets. Hence
$$Z=\zsu_{j=1}^{k}x_{j}{\frac {\zpar} {\zpar x_{j}}}
-\zsu_{j=k+1}^{m-1}x_{j}{\frac {\zpar} {\zpar x_{j}}}
+\ze x_{m}{\frac {\zpar} {\zpar x_{m}}}$$
$\ze =\zmm 1$ on each $A_{p}$.

\subsection{Construction of pseudo-circles}\mylabel{sec-3.2}
Since $\zm$ is non-negative and proper, the $\za$-limit of any
regular trajectory of $Z$ is a (local) minimum or a saddle of $\zm$, whereas
its $\zw$-limit is empty, a (local) maximum or a saddle of $\zm$.
Moreover $Z$ does not possesses any pseudo-circle because no trajectory
of a gradient vector field has its $\za$-limit equal to its $\zw$-limit.
Clearly $Z$ does not have rivets nor topological rivets.

Now by modifying $Z$ we will construct a new vector field with as many
pseudo-circle as minima of $\zm$.

Let $I$ be the set of minima of $\zm$ and $\tilde I$ be that of maxima.
For sake of simplicity let us identify $A_i$ with $B(2r_i )$. Denote by
$E_{i}$, $i\zpe I$, the sphere in $A_{i}$ of radius $r_i$ and center the origin.
For each $i\zpe I$ there exist $\ze_i>0$, $0<r''_i <r_i <r'_i <2r_i$
and a diffeomorphism identifying
${\widetilde A}_{i}=B(r'_i )-{\overline B}(r''_i )$ with
$(-\ze_i ,\ze_i )\zpor S^{m-1}$, endowed with coordinates $(t,y)$,
in such a way that $E_{i}$ corresponds to $\{0\}\zpor S^{m-1}$ and
$Z$ to $\zpar/\zpar t$.

Let $\zf_{i}:{\mathbb R}\zfl{\mathbb R}$ be a function such that $\zf_{i}\big((-\zinf,-\ze_i /2]\big)=1$, $\zf_{i}\big([\ze_i /2,\zinf)\big)=-1$ and $\zf_{i}(t)=0$ if and
only if $t=0$.
On every ${\widetilde A}_{i}\zeq (-\ze_i ,\ze_i )\zpor S^{m-1}$ define
$$Y=\zf_{i}(t)\zpar/\zpar t+(1-\zf_{i}^{2}(t))Z'_i$$ where:

\begin{enumerate}[label={\rm (\arabic{*})}]
\item\mylabel{nuevo1} If $m\zmai 3$ then $Z'_i$ is a limit vector field on
$S^{m-1}$ regarded on $(-\ze_i ,\ze_i )\zpor S^{m-1}$ in the obvious way.
\item\mylabel{nuevo2}  If $m=2$ then
$Z'_{i}=\zl_{i}(t,y)(-y_{2}\zpar/\zpar y_{1}+ y_{1}\zpar/\zpar
y_{2})$ where the function $\zl_{i}\co (-\ze_i ,\ze_i )\zpor S^{1}\zfl{\mathbb R}$
vanishes at some point $(0,q_{i})$ of $\{0\}\zpor S^{1}$ and is positive elsewhere.
\end{enumerate}

Now prolong $Y$ to each $A_i$ by $Z$ inside of $A_{i}-{\widetilde A}_{i}$,
that is on ${\overline B}(r''_i )$, and by $-Z$ outside of $A_{i}-{\widetilde
A}_{i}$, that is on $A_{i}-B(r'_i )$.
In turn, extend $Y$ already defined on $\zung_{i\zpe I}A_{i}$
to the whole manifold $M$ by $-Z$ on $M-\zung_{i\zpe I}A_{i}$

To be sure that $Y$ is $G$-invariant, in each orbit $\mathcal O$
included in $I$ choose
a point $i$ and construct $Y$ on $A_i$ as before. Then by means of the $G$-action construct $Y$ on every $A_j$, $j\zpe\mathcal O$.
As the action of $G$ on $\zung_{j\zpe\mathcal O}A_j$ is free by property
\ref{especial6d} of the family $\{(A_{p},\zr_{p})\}_{p\zpe C}$, this construction
is coherent. Therefore {\em from now on $Y$ will be assumed $G$-invariant}.

Notice that the singularities of $Y$ in $M-\zung_{i\zpe I}E_{i}$
are saddles or sources. The singularities of $Y$ in $\zung_{i\zpe
I}E_{i}$ are never sources nor rivets since each of them is the
$\zw$-limit of two or more regular trajectories traced in
$M-\zung_{i\zpe I}E_{i}$. Thus $Y$ has no topological  rivet and,
consequently, no chain. Besides every $E_i$ contains a single
pseudo-circle of $Y$ {\em denoted by $P_i$ henceforth}; this vector field
does not possess any other pseudo-circle.

It is easily checked that $Y$ is complete with no regular periodic trajectories.
On the other hand the set $Y^{-1}(0)$ of singularities of $Y$ consists of $C$
plus the singularities in each $E_i$ (a finite number for every $E_i$). Since
the family $\{E_{i}\}_{i\zpe C}$ is locally finite because $\{A_{p}\}_{p\zpe C}$
is, it follows that $Y^{-1}(0)$ is discrete and countable. Moreover the set
of sources of $Y$ equals $I\zun \tilde I$.

\begin{lemma}\mylabel{lem-2}
There exists a subset $Q\zco M$, which does not exceed
dimension $m-1$, such that for every point $q\zpe(M-Q)$, the $Y$-trajectory of $q$ is
regular and included in $M-Q$, its $\za$-limit is a source or empty, and
its $\zw$-limit a pseudo-circle.
\end{lemma}

\begin{proof}
As the set of zeros of $Y$ is countable and
$\zung_{i\zpe I}E_{i}$ can be enclosed in dimension $m-1$, it
suffices to consider the points $q$ of $M-\zung_{i\zpe I}E_{i}$
such that $Y(q)\znoi 0$. On the other hand since the outset and the
inset of any saddle is enclosed in dimension $m-1$, the set of points whose
$\za$-limit or whose $\zw$-limit is a saddle is enclosed in dimension $m-1$.

Thus it suffices to study those points $q$ in $M-\zung_{i\zpe I}E_{i}$
whose trajectory is regular and intersects some $A_i$.

By construction $Y$ is tangent to
$E_{i}\zeq \{0\}\zpor S^{m-1}$ and a limit vector field on this submanifold.
Therefore there exists $\{0\}\zpor{\tilde Q}_{i}\zco E_{i}$, which can be enclosed
in dimension $m-2$, such that for any $q\zpe(E_{i}-\{0\}\zpor{\tilde Q}_{i})$ its
$Y$-trajectory has the pseudo-circle $P_i$ as $\zw$-limit. Consequently,
the pseudo-circle $P_{i}$ is the $\zw$-limit of the $Y$-trajectory of each
point of $(-\ze_i ,\ze_i )\zpor(S^{m-1}-{\tilde Q}_{i})$.

Let $\zF_t$ be the flow of $Y$. The set $(-\ze_i ,\ze_i )\zpor{\tilde Q}_{i}$ does not
exceed dimension $m-1$ and, since ${\mathbb Q}$ is countable,
neither does $\zung_{t\zpe\mathbb Q}\zF_{t}((-\ze_i ,\ze_i )\zpor{\tilde Q}_{i})$.
In other words, taking into account that $I$ is countable follows that
the set of points $q\zpe(M-\zung_{j\zpe I}E_{j})$, with
$Y(q)\znoi 0$, whose $Y$-trajectory intersects some $A_i$ but whose
$\zw$-limit is not a pseudo-circle may be enclosed in dimension $m-1$.

Finally if $M-Q$ is not $Y$-saturated, since all points of the trajectory of $q\in M-Q$ have the
same properties, we may replace $Q$ by $Q'=M-M'$ where $M'$ is the $Y$-saturation of $M-Q$.
\end{proof}

\subsection{Construction of chains}\mylabel{sec-3.3}
First notice that $Y$ is tangent to $I_G=I_{G}^{\max}$ because it is
$G$-invariant. Consider a set $Q$ as in Lemma \ref{lem-2}, which is
$Y$-saturated  since $M-Q$ is so. On the other hand as $G$ is finite the set
$G\zpu Q$ still has the properties of Lemma \ref{lem-2}. In short, we may suppose
that $Q$ is $G$ and $Y$-saturated. Since
${\zmontar \zci\over{Q}}={\zmontar \zci\over{I_G}}=\zva$
and $I_G$ is closed, there exists a countable
set $N\zco M-(Q\zun I_G )$ that is dense in $M$ and $G$-saturated. Observe that
any orbit of the action of $G$ on $N$ has $\zlma$ elements because this action is
free. Let $\F$ be the family of all trajectories of $Y$ that intersect $N$.
Observe that $\F$ is infinite since no trajectory of $Y$ is locally dense.

\begin{lemma}\mylabel{lem-3}
Let $U$ be a $G$-invariant vector field on $M$, $\zf_t$ be its flow and
$q$ be a point whose $U$-trajectory is non-periodic. Given
$(g,t),(h,s)\zpe G\zpor\RR$, if $(g\zci\zf_t )(q)=(h\zci\zf_s )(q)$ then
$t=s$ and either $g=h$ or $q\zpe I_G$.
\end{lemma}

\begin{proof}
From hypotheses, it immediately follows that $((h^{-1}g)\zci\zf_{(t-s)})(q)=q$.
Therefore $$q=((h^{-1}g)\zci\zf_{(t-s)})^\zlma (q)
=((h^{-1}g)^\zlma \zci\zf_{\zlma(t-s)})(q)
=\zf_{\zlma(t-s)}(q)$$
since $\zlma$ is the order of $G$. Hence $\zf_{\zlma(t-s)}(q)=q$, which implies
$\zlma(t-s)=0$ and $t=s$.

Thus $(h^{-1}g)\zpu q=q$. If $h^{-1}g\znoi e$ then $q\zpe I_{(h^{-1}g)}\zco I_G$.
\end{proof}

\begin{corollary}\mylabel{cor-1}
The natural action of $G$ on $\F$ is free.
\end{corollary}

\begin{proof}
Assume $g\zpu T=T$ for some $g\zpe G$ and $T\zpe\F$. Then given $q\zpe T$
there exists $t\zpe\RR$ such that $\zF_t (q)=g\zpu q$ and, applying Lemma \ref{lem-3}
to $Y$ and $q$, it follows that $t=0$ and $g\zpu q=q$. As $q\znope I_G$, then $g=e$ must hold.
\end{proof}

The set $\F$ is a disjoint union of $G$-orbits, say $\F_n$, $n\zmai 3$ (by technical reasons we start at natural three). Let $\NN'=\NN-\{0,1,2\}$. Since by Corollary \ref{cor-1}
each $\F_n$ consists of $\zlma$ different trajectories one set
$\F=\{T_{nk}\co n\zpe\NN',\,k=1,\dots,\zlma\}$, where
$\F_n =\{T_{nk}\co k=1,\dots,\zlma\}$, in such a way that $T_{nk}\znoi T_{n'k'}$
if $(n,k)\znoi(n',k')$. (That is to say first one numbers the $G$-orbits in $\F$ and
then, with a second subindex, the elements of every orbit.)

Consider a sequence of $G$-invariant compact sets $\{ K_{n}\}_{n\zpe \mathbb N}$
such that $K_{n}\zco {\zmontar \zci\over{K}}_{n+1}$ and
$\zung_{n\zpe\mathbb N}K_{n}=M$. For every trajectory $T_{nk}$ let $W_{nk}$ be a
set of $n-1$ different points of $T_{nk}$ in such a way that:

\begin{enumerate}[label={\rm (\alph{*})}]
\item\mylabel{especial7a} If $g\zpu T_{nk}=T_{nk'}$ then $g\zpu W_{nk}=W_{nk'}$.
\item\mylabel{especial7b} $W_{nk}\zco M-K_{n}$ if the $\za$-limit of $T_{nk}$
is empty.
\item\mylabel{especial7c} $W_{nk}\zco\zr_{i}^{-1}(B(r_i /n))$ if the
$\za$-limit of $T_{nk}$ is $i$.
\end{enumerate}

Let $W=\zung_{n\zpe\NN',\,k=1,\dots,\zlma}W_{nk}$; then $C\zun W$ is a
countable set whose accumulation points are the minima  and the maxima
of $\zm$, i.e.\ the elements of $I\zun\tilde I$. Therefore
$C\zun W$ is closed and there exists
a function $\zt:M\zfl [0,1]\zco \mathbb R$ such that $\zt^{-1}(0)=C\zun W$.
Set $X=\zt Y$. It easily seen that:

\begin{enumerate}[label={\rm (\arabic{*})}]
\item\mylabel{especial8.1} $X^{-1}(0)=Y^{-1}(0)\zun W$ is countable
and closed. The set of its accumulation points equals $I\zun\tilde I$.
\item\mylabel{especial8.2}  $X$ is complete and has no periodic regular
trajectory.
\item\mylabel{especial8.3} $\{P_{i}\}_{i\zpe I}$ is the family of all
pseudo-circles of $X$.
\item\mylabel{especial8.4}  Let $C_{nk}$ be the family of $X$-trajectories of
$T_{nk}-W_{nk}$ endowed with the order induced by that of $T_{nk}$ as
$Y$-trajectory. Then $C_{nk}$ is a chain of $X$ of order $n$ whose
rivets are the points of $W_{nk}$. Besides $C_{n1},\dots,C_{n\zlma}$ are
the only chains of $X$ of order $n$ and hence $\{C_{nk}\}$, $k=1,\dots,\zlma$,
$n\zpe\NN'$, is the set of all the chains of $X$.

Denote by $H_{nk}$ the last link of $C_{nk}$ and by $P_{\zl(n,k)}$ the
$\zw$-limit of $H_{nk}$ (therefore $\zl$ is a map from
$\NN'\zpor\{1,\dots,\zlma\}$ to $I$).
\item\mylabel{especial8.5} $\zung_{n\zpe\NN',\, k=1,\dots,\zlma}H_{nk}$ is
dense in $M$.
\end{enumerate}

\begin{remark}\mylabel{rem-1}
Notice that the chains $C_{nk}$ given by \ref{especial8.4} can be described
in topological terms as finite sequences of $X$-trajectories such that:

\begin{enumerate}[label={\rm (\roman{*})}]
\item\mylabel{especial9i} Between two consecutive links, the $\zw$-limit of the
first one equals the $\za$-limit of the second one. Moreover this
set consists in a topological rivet.
\item\mylabel{especial9ii}  The $\za$-limit of the first link is an accumulation point of
$X^{-1}(0)$ or empty.
\item\mylabel{especial9iii} The $\zw$-limit of the last link is a pseudo-circle.
\end{enumerate}

Observe that $\{C_{nk}\}$, $k=1,\dots,\zlma$, $n\zpe\NN'$, is the set of all the objects
satisfying \ref{especial9i}, \ref{especial9ii} and \ref{especial9iii} above because $W$ is
the set of topological rivets of $X$ ($Y$ has no rivet). Therefore
{\em any continuous automorphisms of $X$ maps chains to chains}.
\end{remark}

By definition the {\em roll} $R_i$, $i\zpe I$, is the union of all $H_{nk}$
whose $\zw$-limit equals $P_i$.

\subsection{$X$ is a suitable vector field}\mylabel{sec-3.4}
In this subsection $\zF_t$ will be the flow of $X$.
Consider a homeomorphism $f\co M\zfl M$ such that $f\zci\zF_t =\zF_t \zci f$ for any
$t\zpe\RR$.

\begin{proposition}\mylabel{pro-2}
For each roll $R_i$ there exist $t_i \zpe\RR$ and $g_i \zpe G$ such that
$f=g_i \zci\zF_{t_i}$ on $R_i$.
\end{proposition}

\begin{proof}
Fixed a $R_i$ consider a chain $C_{nk}$ whose last link $H_{nk}$ has $P_i$ as
$\zw$-limit. By Remark \ref{rem-1}, $f(C_{nk})$ is a chain of of order $n$, so
$f(C_{nk})=C_{nk'}$ and $f(H_{nk})=H_{nk'}$ for some $k'\zpe\{1,\dots,\zlma\}$.
Moreover $f(P_i )$ is the $\zw$-limit of $H_{nk'}$.

As $\F_n =\{T_{n1},\dots,T_{n\zlma}\}$ is an orbit of the action of $G$ on $\F$,
there exists $h\zpe G$ such that $h\zpu T_{nk}=T_{nk'}$, hence
$h\zpu H_{nk}=H_{nk'}$. Now by composing $f$ on the left with $h^{-1}$ we may
assume $f(C_{nk})=C_{nk}$, $f(H_{nk})=H_{nk}$ and $f(P_i )=P_i$.

Since $f$ commutes with any $\zF_t$ and the regular trajectory of $P_i$ is not
periodic, {\em there exists a single $t_i \zpe\RR$ such that $f=\zF_{t_i}$ on $P_i$}.

Now consider any $H_{ab}$, $(a,b)\zpe\NN'\zpor\{1,\dots,\zlma\}$, with $\zw$-limit
$P_i$. Then $f(H_{ab})$, which is the last link of $f(C_{ab})$, has $P_i$ as
$\zw$-limit too.

Recall that the $G$-orbit of $i$ possesses $\zlma$ elements, that is to say if
$i\znope I_G$. Therefore there is a single $H_{ab'}$ with $\zw$-limit $P_i$;
$H_{ab}$ itself. Indeed, there are only $\zlma$ chains of order $a$ and the
$\zw$-limits of their last links are included in the disjoint union
$\zung_{g\zpe G}(g\zpu A_i )$. In other words $f(H_{ab})=H_{ab}$.

But $H_{ab}$ is a non-periodic regular trajectory and $f$ commutes with the flow
$\zF_t$, so there exists $t'\zpe\RR$ such that $f=\zF_{t'}$ on $H_{ab}$.   As $P_i$
is the $\zw$-limit of $H_{ab}$ one has $f=\zF_{t'}$ on $P_i$ too, hence $t'=t_i$.
\end{proof}

From Proposition \ref{pro-2}, it immediately  follows:

\begin{corollary}\mylabel{cor-2}
For every roll $R_i$ there exist $t_i \zpe\RR$ and $g_i \zpe G$ such that
$f=g_i \zci\zF_{t_i}$ on ${\overline R}_i$.
\end{corollary}

\begin{lemma}\mylabel{lem-4}
The family $\{{\overline R}_i \}_{i\zpe I}$ is locally finite
and $\zung_{i\zpe I}{\overline R}_i =M$.
\end{lemma}

\begin{proof}
For the first part it suffices to show that  $\{{R}_i \}_{i\zpe I}$ is locally finite.
 From the fact that $P_i$ is included in
$A_i$ follows that $\zm(R_{i})$ is low bounded by
$\zm(i)$. But $I$ is a discrete set and $\zm$ a non-negative
proper Morse function, so in every compact set
$\zm^{-1}((-\zinf,a])$ there are only a finite number of elements
of $I$. Therefore $\zm^{-1}((-\zinf,a])$ and of course
$\zm^{-1}(-\zinf,a)$ only intersect a finite number of rolls
$R_i$. Finally, observe that $M=\zung_{a\zpe\mathbb
R}\zm^{-1}(-\zinf,a)$.

By construction of $X$ (see Property \ref{especial8.5})
$\zung_{n\zpe\NN',\, k=1,\dots,\zlma}H_{nk}$ is dense in $M$.
On the other hand $\zung_{i\zpe I}{\overline R}_i$ is closed because
$\{{\overline R}_i \}_{i\zpe I}$ is locally finite, so
$\zung_{i\zpe I}{\overline R}_i$ is a closed set that includes
$\zung_{n\zpe\NN',\, k=1,\dots,\zlma}H_{nk}$.
\end{proof}

\begin{lemma}\mylabel{lem-5}
All scalars $t_i$ given by Corollary \ref{cor-2} are equal.
\end{lemma}

\begin{proof}
Assume that the family $\{t_i \}_{i\zpe I}$ possesses two or more
elements. Fixed one of them, say $t$, set $D_1$ the union of
all ${\overline R}_{i}$ such that $t_{i}=t$ and $D_2$ the union of all
${\overline R}_{i}$ such that $t_{i}\znoi t$. By Lemma \ref{lem-4},
$D_1$ and $D_2$ are closed and $M=D_{1}\zun D_{2}$.

On the other hand if $p\zpe D_{1}\zin D_{2}$ then there exist
${\overline R}_{i}\zco D_1$ and ${\overline R}_{j}\zco D_2$ such that
$p\zpe{\overline R}_{i}\zin{\overline R}_{j}$; so
$f(p)=(g_i \zci\zF_{t_i})(p)=(g_j \zci\zF_{t_j})(p)$. As $t_i \znoi t_j$ from
Lemma \ref{lem-3} applied to $X$ and $p$ follows that the $X$-orbit of $p$
is periodic. Hence $X(p)=0$ since $X$ has
no periodic regular trajectories, which implies that $ D_{1}\zin D_{2}$
is countable. Consequently $M-D_{1}\zin D_{2}$ is connected. But
$M-D_{1}\zin D_{2}=(D_{1}-D_{1}\zin D_{2})\zun(D_{2}-D_{1}\zin D_{2})$
where the terms of this union are non-empty, disjoint and closed in
$M-D_{1}\zin D_{2}$, {\it contradiction}.
\end{proof}

Now composing $f$ with $\zF_{-t}$ where $t$ is the scalar given by
Corollary \ref{cor-2} and Lemma \ref{lem-5} we may assume, without lost
of generality, that $f(x)=g_x \zpu x$ for any $x\zpe M$ where $g_x \zpe G$.
For finishing the proof of the existence of $(g,t)\zpe G\zpor\RR$ such that
$f=g\zci\zF_t$  it suffices to apply the following result:

\begin{lemma}\mylabel{lem-6}
Consider a continuous and injective map $\zt\co M\zfl M$. Assume for every
$x\zpe M$ there exists $g_x \zpe G$ such that $\zt(x)=g_x \zpu x$. Then
$\zt =g$ for some $g\zpe G$.
\end{lemma}

\begin{proof}
Given $g\zpe G$ set $D_g \co=\{p\zpe M-I_G \co\zt(p)=g\zpu p\}$. As the space
is Hausdorff $G_g$ is closed in $M-I_G$. Moreover $D_g \zin D_h=\zva$ when $g\znoi h$
since the action of $G$ on $M-I_G$ is free. As $G$ is finite and
$M-I_G =\zung_{h\zpe G}D_h$, every $D_g$ is open too so
union of connected components of $M-I_G$.

Some of the sets $D_g$, $g\zpe G$, has to be non-empty and by composing $\zt$ on
the left with a suitable element of $G$ one may assume $D_e\znoi\zva$.
Let us see that $\overline D_e$ is open in $M$. If so the proof is finished since
$M$ is connected and necessarily $\overline D_e =M$.

Consider any $p\zpe\overline D_e$. If $p\zpe M-I_G$ then $p\zpe D_e$ and it is an
interior point. Now assume $p\zpe I_h =I_{h}^{\max}$ for some $h\zpe G-\{e\}$.
Then around $p$ there exist coordinates $(x_1 ,\dots,x_m )$ whose domain $D$ is
diffeomorphic through these coordinates to an open ball $B(r)$ such that
$p\zeq 0$ and $h$ is given by the symmetry $\zG(x)=(x_1 ,\dots, x_{m-1},-x_m )$.

Let $S^+$ be the open half domain defined by $x_m >0$ and $S^-$ that given by
$x_m <0$. If $r$ is sufficiently small then $S^+ \zun S^- \zco M-I_G$, and
$I_G \zin D$ and $I_h \zin D$ are equal and defined by $x_m =0$ (see \ref{especial4} of
Lemma \ref{lem-1}). As $p\zpe\overline D_e$ and $D_e$ is an union of connected
components of $M-I_G$ necessarily one at least of the foregoing open half domains is
included in $D_e$. Assume $S^+ \zco D_e$, the other case is similar.
If $S^- \znoco D_e$ then $S^- \zin D_e =\zva$ and there is some $\bar g\zpe G-\{e\}$
such that  $S^- \zco D_{\bar g}$ because $S^-$ is connected.

By continuity $\bar g\zpu p=\zt(p)$ and $e\zpu p=\zt(p)$, hence $\bar g\zpu p=p$
and by \ref{especial4} of Lemma \ref{lem-1} one has $\bar g=h$. Therefore $\zt$ on
$S^-$ equals $h$. Thus
$\zt(0,\dots,0,-\ze)=(0,\dots,0,\ze)=e\zpu(0,\dots,0,\ze)=\zt(0,\dots,0,\ze)$, $\ze>0$,
and $\zt$ is not injective, {\em contradiction}.
In short $S^- \zco D_e$ and necessarily $p\zpe D\zco\overline D_e$.
\end{proof}

Finally, if $g\zci\zF_t =Id_M$ then $Id_M =(g\zci\zF_t )^\zlma
=g^\zlma \zci\zF_{\zlma t}=\zF_{\zlma t}$. As $X$ has regular non-periodic
trajectories $t=0$, so $g=e$. This fact implies the injectivity of the morphism from
$G\zpor\RR$ to $\Aut_{0}(X)$. Therefore {\em  the main theorem is proved under
the supplementary hypothesis $I_G =I_G^{\max}$.}

\begin{remark}\mylabel{rem-2}
Consider a function $\zf\co M\zfl\RR$ that is $G$-invariant, positive and
bounded. Then $\zf X$ is a complete vector field. Besides $X$ and $\zf X$ have
the same trajectories (with different speeds but the same orientation by the time),
$\za$ and $\zw$-limits, pseudo-circles, rolls, rivets and chains. Therefore reasoning
as before but this time with $\zf X$ shows that
$(g,t)\zpe G\zpor\RR\zfl g\zci{\tilde \zF}_t\zpe\Aut_{0}(\zf X)$ is a group isomorphism
where ${\tilde \zF}_t$ is the flow of $\zf X$.

In other words $\zf X$ is a suitable vector field too.
\end{remark}

\section{The general case}\mylabel{sec-4}
In this section the main result will be proved in the general case by reducing it to
the almost free one.

Given $g\zpe G-\{e\}$ let $J_g$ be the set of those points $p\zpe I_g$ such that the
dimension of $I_g$ at $p$ is $\zmei m-2$. Set  $J_G\co=\zung_{g\zpe G-\{e\}}J_g$. It is
easily seen that $J_G$ is a $G$-invariant closed set and $M-J_G$ a $G$-invariant dense
open set.

One will say that {\em the dimension of a point $p\zpe J_G$ is zero} if the dimension
at $p$ of every $I_g$ such that $p\zpe J_g$, is zero. Let $S_0$ be the set of all points
of $J_G$ of dimension zero. Clearly $S_0$ is $G$-invariant and $S_0 \zco J_G$.

By making use of normal coordinates with respect to a $G$-invariant Riemannian metric,
centered at points of $J_G$, it is easily checked that:

\begin{enumerate}[label={\rm (\arabic{*})}]
\item\mylabel{especial9.1} $S_0$ has no accumulation point so it is countable and closed.
\item\mylabel{especial9.2} Every $q\zpe S_0$ possesses a neighborhood whose
intersection with $J_G$ equals $\{q\}$.
\item\mylabel{especial9.3} For any $q\zpe J_G -S_0$ and any neighborhood $A$ of $q$
the set $A\zin J_G$ is uncountable.
\end{enumerate}

Since $M$ is paracompact (even more $\zs$-compact) from \ref{especial9.1} and
\ref{especial9.2} follows the existence of a locally finite
family of disjoint open sets $\tilde\A =\{\tilde A_q\}_{q\zpe S_0}$ such that every
$\tilde A_q \zin S_0=\{q\}$. As $S_0$  is $G$-invariant and $G$ is finite shrinking the
elements of $\tilde\A$ allows us to assume  that this family is $G$-invariant.
Even more one can suppose that each $\tilde A_q$ is a domain of normal coordinates
centered at $q$ of some $G$-invariant Riemannian metric.

For every $q\zpe S_0$ consider a set $q\zpe C_q \zco \tilde A_q$ that in the normal coordinates
mentioned before is a closed non-trivial segment sufficiently small.
Set $D_q \co =\{(g,p)\zpe G\zpor S_0 \co g\zpu p=q\}$ and
$E_q \co =\zung_{(g,p)\zpe D_q}g\zpu C_p$. Then $E_q \zco\tilde A_q$ so the family
$\{E_q \}_{q\zpe S_0}$ is locally finite, hence $S_1 \co=\zung_{q\zpe S_0}E_q$ is
closed. Moreover $S_1$ is $G$-invariant and any neighborhood of any point of $S_1$
includes uncountably many points of $S_1$. On the other hand we may assumed that
$M-S_1$ is connected without loss  of generality.

Thus the set  $\tilde M\co =M-(J_G \zun S_1 )$ is $G$-invariant, connected,
dense and open, since $S_1$ can be enclosed in dimension one and $J_g$ in dimension $m-2$
and $M-S_1$ was connected. Besides each neighborhood of every point of $J_G \zun S_1$
contains uncountably many elements of $J_G \zun S_1$.

On the other hand the action of $G$ on $\tilde M$ is almost free. Indeed, if
$q\zpe I_g \zin \tilde M$ for some $g\zpe G-\{e\}$ then $q\znope J_g$ so the dimension
of $I_g$ at $q$ equals $m-1$ and $q\zpe I_{g}^{\max}$.

By \cite[Proposition 5.5]{TV} there exists a bounded function $\zf\co M\zfl\RR$, which
is positive on $\tilde M$ and vanishes on $M-\tilde M$, such that the vector field $\hat X$
on $M$ defined by $\hat X=\zf X$ on $\tilde M$ and  $\hat X=0$ on $M-\tilde M$ is
differentiable.

Notice that $g_{*}^{-1} (\hat X)$ equals $(\zf\zci g) X$ and zero on $M-\tilde M$. Therefore
by  taking $\zlma^{-1}\zsu_{g\zpe G} (\zf\zci g)$ instead of $\zf$ we my assume that $\zf$ is
$G$-invariant, which implies that $\hat X$ is $G$-invariant too.

Given a singularity $p$ of $\hat X$ one has two possibilities:

\begin{enumerate}[label={\rm (\alph{*})}]
\item\mylabel{especial10a} Any neighborhood of $p$ contains uncountably many
zeros of $\hat X$; that is to say $p\zpe M-\tilde M$.
\item\mylabel{especial10b} There is a neighborhood of $p$ that only includes countably
many zeros of $\hat X$; that is to say $p\zpe\tilde M$ and $X(p)=0$.
\end{enumerate}

Consider $f\zpe\Aut_0 (\hat X)$. From \ref{especial10a} and \ref{especial10b}
follows that $f(M-\tilde M)=M-\tilde M$ and $f(\tilde M)=\tilde M$.  Thus the
restriction of $f$ to $\tilde M$ is a continuous automorphism of $\hat X_{\zbv M}=\zf X$
and, by Section \ref{sec-3}, there exist $g\zpe G$ and $t\zpe\RR$ such that
$f=g\zci\hat\zF_t$ on $\tilde M$ where $\hat\zF_t$ is the flow of $\hat X$.
By continuity $f=g\zci\hat\zF_t$ everywhere. The uniqueness of $g$ and $t$ is obvious.
{\em In short the main result is proved in the general case.}

\section{Actions on manifolds with boundary}\mylabel{sec-5}
Let $P$ be a $m$-manifold with nonempty boundary  $\zpar P$. Then each homeomorphism
$f\co P\zfl P$ induces a homeomorphism $f\co \zpar P\zfl\zpar P$. Therefore the same
reasoning as in Section 4 of \cite{TV} shows that the main result of the present paper
also holds for a connected manifold $P$, of dimension $m\zmai 2$, with
nonempty boundary and a finite subgroup $G$ of  $\Diff(M)$.

\section{Examples}\mylabel{sec-6}

\begin{example}\mylabel{exaPla}
{\rm On $\RR^2$ consider the group of two elements $G=\{e,g\}$ where $g(x)=(-x_1 ,x_2 )$.
Then the action of $G$ on $\RR^2$ is almost free and $I_G^{\max}=\{0\}\zpor\RR$. For
constructing a suitable vector field $X$ as in Section \ref{sec-3}, one can start with the Morse
function $\zm=(x_1^2 -1)^2 +x_2^2$ that has two minima at $(1,0)$ and $(-1,0)$ respectively
and a saddle at the origin.

Therefore at the end of the process $X$ has two pseudo-circles around $(1,0)$ and $(-1,0)$
respectively. Moreover the set of singularities of $X$ is countable and accumulates towards
$(1,0)$, $(-1,0)$ and the infinity. Observe that $I_G^{\max}$ consists of a singular point and
two regular trajectories with the singular point as $\za$-limit and empty $\zw$-limit.

In a similar way, on $S^2$ one may consider the group $G=\{e,g\}$ where now
$g(x)=(x_1 ,x_2 ,-x_3 )$ and the Morse function $\zm=2x_1^2 +x_2^2$ that has two minima
at $(0,0,\zmm 1)$, two maxima at $(\zmm 1,0,0)$ and two saddles at $(0,\zmm 1,0)$. The
action of $G$ is almost free and there is no minimum on
$I_G^{\max}=S^2 \zin(\RR^2 \zpor\{0\})$.}
\end{example}

\begin{example}\mylabel{exaMor}
Let $M$ be a connected compact manifold of dimension $m\zmai 2$. Given $G$, a finite group
of diffeomorphisms of $M$, and a $G$-invariant Morse function $\zm$, let $X$ be the gradient
vector field of $\zm$ with respect to a $G$-invariant Riemannian metric. Then, although the group of smooth automorphisms of $X$, namely $\Aut(X)$, may equal $G\zpor\RR$ (e.g.\ in \cite[Example 5.2]{TV}), the group
$\Aut_{0}(X)$ is strictly greater than $G\zpor\RR$.

Indeed, first note that there always exist a minimum and a maximum of $\zm$, that we  denote by  $p$ and $q$, and a
trajectory $\zg$ of $X$ whose $\za$-limit and $\zw$-limit are $p$ and $q$ respectively. Consider a closed sufficiently small
$(m-1)$-disk $D$ transverse to $X$ and intersecting $\zg$ just once. We may suppose, without lost of
generality, that every trajectory of $X$ intersects $D$ at most once and if so its $\za$-limit
equals $p$ and its $\zw$-limit $q$.

Let $E$ be the set of those points of $M$ whose trajectory meets $D$. Then $E$ is diffeomorphic
to $D\zpor\RR$  in such a way that $X$ becomes $\zpar/\zpar s$ where $D\zpor\RR$ is endowed
with coordinates $(x,s)=(x_1 ,\dots,x_{m-1},s)$.

For each continuous function $\zl\co D\zfl\RR$ such that $\zl(\zpar D)=0$ one defines
$f\zpe\Aut_{0}(X)$ to be $f=Id$ on $M-E$ and $f(x,s)=\zF_{\zl(x)}(x,s)=(x,s+\zl(x))$ on $E$, where
$\zF_t$ is the flow of $X$. As $\overline E -E=((\zpar D)\zpor\RR)\zun\{p,q\}$ our $f$ is continuous.
Its inverse is given by $-\zl$ and obviously $f$ is an automorphism of $X$, which in general does
not belong to $G\zpor\RR$.

Other way for constructing such a $f$ is to consider a homeomorphism $\zt\co D\zfl D$ with
$\zt_{\zbv\zpar D}=Id_{\zbv\zpar D}$ and set $f(x,s)=(\zt(x),s)$ on $E$, $f=Id$ elsewhere.

Observe that an analogous construction can be done if the gradient field is slightly modified,
namely if one adds a finite number of new singularities of index zero. Thus, in general, the group
of continuous automorphisms of vector fields constructed in \cite{TV} is strictly greater than
$G\zpor\RR$ (for the non-compact case the reasoning above can be easily adapted if there is
at least a maximum).  In other words, \emph{these vector fields determine $G$ in the smooth
category but not in the continuous one.}
\end{example}

The next two examples are extensions of Corollary \ref{mainco} to situations where
differentiability at every point is not assured. One has:

\begin{proposition}\mylabel{pro-3}
Consider a finite group $G$ of homeomorphisms of a connected $C^\zinf$ manifold $M$ of
dimension $\zmai 2$. Let $A$ be a $G$-invariant open set of $M$. Assume that:
\begin{enumerate}[label={\rm (\alph{*})}]
\item\mylabel{pro-3a} $A$ is connected and dense, and every neighborhood of each point
of $M-A$   contains uncountably many points of $M-A$.
\item\mylabel{pro-3b} Under restriction each element of $G$ is a diffeomorphism of $A$.
\end{enumerate}

Then there exists a $G$-invariant smooth, and therefore continuous, flow
 $\zq\co\RR\zpor M\zfl M$  such that the map
 $$\begin{array}{rcl}
G\zpor\RR & \to & \Aut_{0}(\zq)\\
(g,t) & \mapsto & g\zci\zq_t
\end{array}$$
is a group isomorphism.
\end{proposition}

\begin{proof}
Consider $G$ under restriction as a group of diffeomorphisms of $A$ and define $J_G$ and
$S_1$ as in Section \ref{sec-4} (for $A$ of course). Set
$\tilde M=M-(J_G \zun S_1 \zun(M-A))=A-J_G \zun S_1$. Then the action of $G$ on $\tilde M$
is almost free, which gives rise to a suitable vector field $X$ on $\tilde M$.

Since $J_G \zun S_1 \zun(M-A)$ is a closed set of $M$ with empty interior, a vector field
$\hat X$ on $M$ that
\begin{enumerate}[label={\rm (\arabic{*})}]
\item\mylabel{dem1} vanishes in $J_G \zun S_1 \zun(M-A)$ and
\item\mylabel{dem2} on $\tilde M$  equals $\zf X$ for a suitable function
$\zf\co M\zfl\RR$,
\end{enumerate}
can be constructed as in Section \ref{sec-4}.

The flow of $\hat X$ has the required properties. Indeed, any neighborhood of any point of
$J_G \zun S_1 \zun(M-A)$ includes uncountably many points of this set, and one can
reason as in the second part of Section \ref{sec-4}.
\end{proof}

\begin{example}\mylabel{exaSph}
{Following the notation in \cite[pp.\ 23--24]{HIZ}, in $\CC^4$ endowed with coordinates $z=(z_0 ,z_1 ,z_2 ,z_3 )$ the equations
\begin{gather*}
z_0^3 +z_1^2 +z_2^2 +z_3^2 =0\\
\zsu_{k=0}^3 z_k \bar z_k =1
\end{gather*} define a
smooth real submanifold that is diffeomorphic to the standard sphere $S^5$ such that
 $\beta(z)=(e^{2\zp i/3}z_0 ,z_1 ,z_2 ,z_3 )$ defines a diffeomorphism of
$S^5$ of order three, whose set of fixed points is (diffeomorphic to) $\RR P^3$ \cite[REMARKS p.\ 24]{HIZ}.

The (topological) suspension of $S^5$ and that of $\beta$ give rise to a homeomorphism
$f\co S^6\zfl S^6$ of order three, whose set of fixed points is (homeomorphic to) the
suspension of $\RR P^3$ in such a way that the vertices are the poles. Therefore
$f$ cannot be smoothed otherwise the suspension of $\RR P^3$ has to be a differentiable
manifold, which is not the case.

Let $G$ be the group of homeomorphisms of $S^6$ spanned by $f$, whose order equals
three. Clearly $G$ cannot be smoothed. However, away of the poles $G$ is a group of
diffeomorphisms.

Consider a meridian (that is the intersection of $S^6 \zco\RR^7$ with a plane passing
through the origin and the poles) and saturate it under the action of $G$ for constructing
a $G$-invariant compact set $C$. Finally set $A=S^6 -C$ and apply Proposition \ref{pro-3}
for concluding that, even if $G$ cannot be smoothed, there exists a differentiable flow on
$S^6$ that determines $G$.}
\end{example}

\begin{theorem}\mylabel{exaTop}
Let $G$ be a finite group of homeomorphisms of a connected compact topological $4$-manifold
$M$ with no boundary. Assume that the action of $G$ is free. Then there exist a continuous flow $\widetilde\zF$
that determines $G$.
\end{theorem}

We devote the rest of this section to the proof of Theorem \ref{exaTop}.

Let $P=M/G$ be the topological
quotient manifold and $\zp\co M\zfl P$ the canonical projection, which is a covering.

Fix a point $a$ of $P$. Then $P'\co=P-\{a\}$ has a structure of smooth manifold (Quinn \cite{QU}).
The pull-back of this structure defines a smooth structure in
$M'\co=M-\zp^{-1}(a)=\zp^{-1}(P')$ in such a way that the natural action of $G$ on $M'$
is smooth and $\zp\co M'\zfl P'$ is a smooth covering.

Consider a set $a\zpe C\zco P$ that with respect to some  topological coordinates centered at $a$
is a closed non-trivial segment sufficiently small; assume that $a$ is one of its vertices.
Then $\zp^{-1}(C)$ is a disjoint union of compact
sets $C_1 \zun\dots\zun C_\zlma$ where $\zlma$ is the order of $G$ and each
$\zp\co C_j \zfl C$ a homeomorphism. Notice that $P-C$ is connected and dense in
$P$, and $M-\zp^{-1}(C)$ is connected, dense in $M$ and $G$-invariant.

Now construct a suitable vector field $X$ on $M-\zp^{-1}(C)$. Since
$\zp^{-1}(C-\{a\})$ is closed in $M'$, a vector field $\widehat X$ that vanishes on
$\zp^{-1}(C-\{a\})$ and equals $\zf X$ on $M-\zp^{-1}(C)$ for a suitable function
$\zf\co M'\zfl\RR$ can be constructed as in Section \ref{sec-4}.

On the other hand the flow $\widehat\zF$ of $\widehat X$ can be extended into a continuous flow
$\widetilde\zF\co\RR\zpor M\zfl M$ by setting $\widetilde\zF(\RR\zpor\{b\})=b$ for every
$b\zpe\zp^{-1}(a)$. Accept this fact by the moment; we will prove it later on. If
$f\co M\zfl M$ is a continuous automorphism of $\widetilde\zF$, then
$f(\zp^{-1}(C))=\zp^{-1}(C)$ by the same reason as in Section \ref{sec-4} (replace
singularities of $\widehat X$ by stationary points of the flow $\widetilde\zF$ and take into account
that any neighborhood of any point of $\zp^{-1}(C)$ contains uncountably many points of
$\zp^{-1}(C)$). Therefore $f\co M-\zp^{-1}(C)\zfl M-\zp^{-1}(C)$ is a continuous
automorphism of $\widehat X$ and hence $f=g\zci\widehat\zF_t =g\zci\widetilde\zF_t$ on
$M-\zp^{-1}(C)$ for some $g\zpe G$ and $t\zpe\RR$. Since $M-\zp^{-1}(C)$ is dense
in $M$, by continuity  $f=g\zci\widetilde\zF_t$ everywhere.

Clearly if  $g\zci\widetilde\zF_t =Id$ then $g=e$ and $t=0$.

In short, {\em even if $M$ has no smooth structure, the continuous flow $\widetilde\zF$
determines $G$.}

Let us prove that $\widetilde\zF$ is a continuous flow. The only difficult point is the continuity.
Note that as $\widehat\zF$ is $G$-invariant there is a smooth flow $\zq\co\RR\zpor P'\zfl P'$
such that $\zp\zci\widehat\zF=\zq\zci(Id\zpor\zp)$ (it is the flow associated to the projection of
$\widehat X$  onto $P'$). Denote by $\widetilde\zq\co\RR\zpor P\zfl P$ the extension of $\zq$ defined
by setting $\widetilde\zq(\RR\zpor\{a\})=a$. Observe that $\zp\zci\widetilde\zF=\widetilde\zq\zci(Id\zpor\zp)$.

\begin{lemma}\mylabel{lem-7}
$\widetilde\zq$ is continuous.
\end{lemma}

\begin{proof}
As before the difficult point is the continuity. For checking it one will show that
$\widetilde\zq\co[a,b]\zpor P\zfl P$ is continuous for any $a<b$ belonging to $\RR$.

Consider a map $s\co E_1 \zfl E_2$ between locally compact but not compact topological
spaces. Denote by $\A(E_k )$, $k=1,2$, the Alexandroff compactification of $E_k$ and by
$\A(s)\co\A(E_1 )\zfl\A(E_2 )$ the extension of $s$ that maps the infinity point of $\A(E_1 )$
to that of $\A(E_2 )$. Recall that $\A(s)$ is continuous if and only if $s$ is proper.

The map $h\co[a,b]\zpor P'\zfl [a,b]\zpor P'$ given by $h(t,x)=(t,\zq(t,x))$ is a homeomorphism and
hence $\A(h)\co\A([a,b]\zpor P')\zfl\A([a,b]\zpor P')$ is continuous.

In turn the second projection $\zp_2 \co[a,b]\zpor P'\zfl P'$ is proper, so
$\A(\zp_2 )\co\A([a,b]\zpor P')\zfl\A(P')$ is continuous.

Since the Alexandroff compactification is the smallest one among the Hausdorff compactifications,
the map $g\co[a,b]\zpor\A(P')\zfl\A([a,b]\zpor P')$ that equals the identity on
$[a,b]\zpor P'$ and maps $[a,b]\zpor\{\zinf\}$ to $\zinf$ is continuous.

Finally if one identifies $P$ to $\A(P')$ by regarding $a$ like the infinity point, then
$\widetilde\zq=\A(\zp_2 )\zci\A(h)\zci g$.
\end{proof}

\begin{corollary}\mylabel{cor-3}
$\widetilde\zF$ is continuous.
\end{corollary}

\begin{proof}
Consider the map $l \co\RR\zpor M\zfl P$ given by $l(t,x)=\widetilde\zq(t,\zp(x))$ and
a point $v\zpe M'$. As $\RR$ is contractile, then $l$ is homotopic to the map
$$\begin{array}{rcl}
\RR\zpor M & \to & P\\
(t,x) & \mapsto & \zp(x)
\end{array}$$
Therefore
$l_\sharp \big(\zp_1 (\RR\zpor M,(0,v))\big)=\zp_\sharp\big(\zp_1 (M,v)\big)$ and hence, with respect to the
covering $\zp\co M\zfl P$, there exists
a lift $L\co\RR\zpor M\zfl M$ of $l$, with initial condition $L(0,v)=v$. Moreover
$L=\widetilde\zF$ on $\RR\zpor M'$ since one knows that $\widetilde\zF$ is continuous
on $\RR\zpor M'$ and $\zp\zci\widetilde\zF=\widetilde\zq\zci(Id\zpor\zp)$.

Finally as $\widehat X$ vanishes on $\zp^{-1}(C-\{a\})$, it follows that $L=\widetilde\zF=Id$ on
$\RR\zpor\zp^{-1}(C-\{a\})$. By continuity $L(\RR\zpor\{b\})=b$ for each
$b\zpe\zp^{-1}(a)$. Thus $L=\widetilde\zF$ everywhere.
\end{proof}

\section{Appendix} \mylabel{sec-A}

In this section $B(r)$, $r>0$, will be the ball in $\RR^m$, endowed with coordinates
$(x_1 ,\dots,x_m )$, of center the origin and radius
$r$, and $\zG\co\RR^m \zfl\RR^m$ the symmetry given by
$\zG(x_1 ,\dots,x_m )=(x_1 ,\dots,x_{m-1},-x_m )$.

\begin{lemma}\mylabel{leA-1}
Consider a function $\zm$ defined around the origin of $\RR^m$ and $\zG$-invariant.
Assume that the origin is a non-degenerated singularity. Then about the origin there exist
coordinates $(y_1 ,\dots,y_m )$ such the coordinates of the origin are still $(0,\dots,0)$,
$$\zm=\zsu_{j=1}^{k}y_{j}^{2}
-\zsu_{j=k+1}^{m-1}y_{j}^{2}+\ze y_{m}^2 +\zm(0)$$
where $\ze=\zmm 1$, and $\zG(y_1 ,\dots,y_m )=(y_1 ,\dots,y_{m-1},-y_m )$.
\end{lemma}

\begin{proof}
As $\zm$ is $\zG$-invariant its restriction to the hyperplane $H$ defined by $x_m =0$
has a non-degenerated singularity at the origin. Therefore coordinates $(x_1 ,\dots,x_m )$
can be replaced by coordinates $(y_1 ,\dots,y_{m-1},x_m )$ in such a way that
$\zG(y_1 ,\dots,y_{m-1},x_m )=(y_1 ,\dots,y_{m-1},-x_m )$ and
$$\zm(y_1 ,\dots,y_{m-1},0)=\zsu_{j=1}^{k}y_{j}^{2}
-\zsu_{j=k+1}^{m-1}y_{j}^{2}+\zm(0).$$

On the other hand from the Taylor expansion in variable $x_m$ transversely
to $H$ it follows
$$\zm(y_1 ,\dots,y_{m-1},x_m )=\zm(y_1 ,\dots,y_{m-1},0) \hskip 3truecm$$
$$\hskip 2.5truecm +x_m {\frac {\zpar\zm} {\zpar x_m}}(y_1 ,\dots,y_{m-1},0)
+x_{m}^2 f(y_1 ,\dots,y_{m-1},x_m ).$$
\vskip .2truecm

By the $\zG$-invariance ${\frac {\zpar\zm} {\zpar x_m}}
(y_1 ,\dots,y_{m-1},0)=0$ and $f(y_1 ,\dots,y_{m-1},-x_m )
= f(y_1 ,\dots,y_{m-1},x_m )$. Moreover
$2f(0)={\frac {\zpar^2 \zm} {\zpar x_{m}^2}}(0)\znoi 0$ since the origin
is a non-degenerated singularity. Therefore close to the origin
$(y_1 ,\dots,y_{m-1},y_m )$, where $y_m =x_m \zbv
f(y_1 ,\dots,y_{m-1},x_m )\zbv^{1/2}$, is a system of coordinates as
required.
\end{proof}

\begin{proposition} \mylabel{prA-1}
For every $r>0$  there exists a Morse function $\zt\co\RR^m \zfl\RR$ such that:

\begin{enumerate}[label={\rm (\alph{*})}]
\item\mylabel{especialAa} $\zt$ is $\zG$-invariant.

\item\mylabel{especialAb} If $p$ is a minimum of $\zt$ then $\zG(p)\znoi p$,
that is to say $p$ does not belong to the hyperplane $x_m =0$.

\item\mylabel{especialAc} $\zbv\zt(x)\zbv\zmei \zdbv x\zdbv^2$ on $\RR^m$ and
$\zt(x)=\zdbv x\zdbv^2$ on $\RR^m -B(r)$.
\end{enumerate}
\end{proposition}

For proving the foregoing proposition we need:

\begin{lemma}\mylabel{leA-2} There exists a smooth function $\zr\co\RR\zfl\RR$
such that

\begin{enumerate}[label={\rm (\alph{*})}]
\item\mylabel{especialBa} $\zr(t)=1$ if $t\zmai 1$, $\zr(t)=-1$ if $t\zmei 0$,
$\zr(1/2)=0$, $0<\zr<1$ on $(1/2,1)$ and $-1<\zr<0$ on $(0,1/2)$. Moreover
$\zr(1-t)=-\zr(t)$, $t\zpe\RR$, that t is to say $\zr$ is anti-symmetrical with
respect to $t=1/2$.
\item\mylabel{especialBb} $\zr\,'\zmai 0$ on $\RR$ and $\zr\,'>0$ on $(0,1)$.
Moreover $\zr\,'(1-t)=\zr\,'(t)$, $t\zpe\RR$.
\item\mylabel{especialBc} $\zr\,''>0$ on $(0,1/2)$, $\zr\,''<0$ on $(1/2,1)$ and
$\zr\,''=0$ on $(\RR-(0,1))\zun\{1/2\}$. Moreover
 $\zr\,''(1-t)=-\zr\,''(t)$, $t\zpe\RR$.
\end{enumerate}
\end{lemma}

\begin{proof} Let $\zf$ be a smooth function meeting the requirements of
\ref{especialBc}. Denote by $\zf_1$ its primitive with initial condition $\zf_1  (0)=0$
and by $\zf_2$ the primitive of $\zf_1$ such that $\zf_2 (1/2)=0$.

Then $\zf_2$ is constant and positive on $[1,\zinf)$ while it is constant and negative
on $(-\zinf,0]$. Moreover $\zf_2 (0)=-\zf_2 (1)$. The function
$\zr=(\zf_2 (1))^{-1}\zf_2$ meets the requirements of the lemma (draw the graphics
of $\zf$, $\zf_1$ and $\zf_2$).
\end{proof}

\begin{corollary}\mylabel{coA-1} Consider the function $\zl$ defined by $\zl(t)=t\zr(t)$.
If $\zl'(c)=0$ then $c\zpe (0,1/2)$ and $\zl''(c)>0$.
\end{corollary}

\begin{proof}[Proof of Proposition \ref{prA-1}]
First observe that if $\zt$ is like in Proposition \ref{prA-1} for $r=1$, then for any other
$r>0$ it suffices to take $\zt_r  (x)=r^2 \zt(r^{-1}x)$.

Consider a function $\zr$ as in Lemma \ref{leA-2} and set
$\zt(x)=x^{2}_1 +\dots+x^{2}_{m-1}+x^{2}_{m}\zr(\zdbv x\zdbv^2)$. Then
$\zbv\zt(x)\zbv\zmei\zdbv x\zdbv^2$ everywhere and
$\zt(x)=\zdbv x\zdbv^2$ if $\zdbv x\zdbv\zmai 1$.
On the other hand an elementary computation making use of Corollary \ref{coA-1}
shows that the singularities of $\zt$ are always non-degenerate and belong to the
last axis, while the origin is a saddle.
\end{proof}

\begin{remark}\mylabel{reA-1}
As in one variable  between two consecutive minima
there always exists a maximum, the function $\zl$ of Corollary \ref{coA-1} has a single
singularity, which is a minimum. A more careful computation shows that function $\zt$ of the
proof of Proposition \ref{prA-1} has just three singular points: a saddle and two minima.
\end{remark}


\end{document}